\documentclass{amsart}

\usepackage{amsmath}
\usepackage{amsfonts}
\usepackage{amssymb}
\usepackage{graphicx}
\usepackage{hyperref}
\usepackage{mathrsfs}
\usepackage{pb-diagram}
\usepackage{verbatim}

\newtheorem{theorem}{Theorem}[section]
\newtheorem{lemma}[theorem]{Lemma}

\newtheorem{prop}[theorem]{Proposition}
\newtheorem{defn}[theorem]{Definition}

\newtheorem{remark}[theorem]{Remark}

\numberwithin{equation}{section}

\newcommand{\conn}{\nabla}

\newcommand{\der}{\mathrm{d}\,}
\newcommand{\dd}[1]{\frac{\partial}{\partial #1}}
\newcommand{\dbar}{\bar{\partial}}

\newcommand{\pairing}[2]{\left( #1\, , \, #2 \right)}

\newcommand{\integer}{\mathbb{Z}}
\newcommand{\Z}{\mathbb{Z}}

\newcommand{\R}{\mathbb{R}}
\newcommand{\C}{\mathbb{C}}
\newcommand{\real}{\mathbb{R}}
\newcommand{\cpx}{\mathbb{C}}

\newcommand{\N}{\mathbb{N}}

\newcommand{\bF}{\mathbf{F}}
\newcommand{\cP}{\mathcal{P}}

\newcommand{\consti}{\mathbf{i}\,}
\newcommand{\conste}{\mathbf{e}}
\newcommand{\bi}{\mathbf{i}\,}

\newcommand{\Image}{\mathrm{Im}}
\newcommand{\Ker}{\mathrm{Ker}}

\newcommand{\rIm}{\mathrm{Im}}

\newcommand{\cL}{\mathcal{L}}

\newcommand{\FT}{\mathrm{FT}}
\newcommand{\sign}{\mathrm{sign}}

\begin{document}

\title[non-K\"ahler SYZ]{non-K\"ahler SYZ mirror symmetry}

\author[S.-C. Lau]{Siu-Cheong Lau}
\address{Harvard University}
\email{s.lau@math.harvard.edu}

\author[L.-S. Tseng]{Li-Sheng Tseng}
\address{University of California, Irvine}
\email{lstseng@math.uci.edu}

\author[S.-T. Yau]{Shing-Tung Yau}
\address{Harvard University}
\email{yau@math.harvard.edu}

\begin{abstract}
We study SYZ mirror symmetry in the context of non-K\"ahler Calabi-Yau manifolds.  In particular, we study the six-dimensional Type II supersymmetric $SU(3)$ systems with Ramond-Ramond fluxes, and generalize them to higher dimensions.  
We show that Fourier-Mukai transform provides the mirror map between these Type IIA and Type IIB supersymmetric systems in the semi-flat setting.  This is concretely exhibited by nilmanifolds.
\end{abstract}

\maketitle

\section{Introduction}

The Strominger-Yau-Zaslow (SYZ) approach \cite{SYZ96} to mirror symmetry has brought many fruitful results in the past two decades.  For a mirror pair of Calabi-Yau manifolds $(X, \check{X})$, it asserts that there exist special Lagrangian fibrations $\mu:X\to B$ and $\check{\mu}:\check{X}\to B$ such that $\mu^{-1} (\{b\})$ and $\check{\mu}^{-1} (\{b\})$ are dual tori.  As a consequence, Lagrangian sections of $\mu$ should correspond to holomorphic line bundles on $\check{X}$, which explains the interchange between symplectic geometry of $X$ and complex geometry of $\check{X}$.  Gross \cite{Gross-topo} gave a verification of the above picture in the topological level.  The Gross-Siebert program \cite{GS07} gave a sophisticated algebraic formulation of the SYZ construction.  Moreover, the SYZ program was successfully carried out in various situations for K\"ahler manifolds \cite{LYZ,Leung,Auroux07,Cho-Oh,CL,FOOO1,CLL,AAK}.

This paper explores the SYZ approach for non-K\"ahler Calabi-Yau manifolds, i.e. non-K\"ahler almost Hermitian manifolds with $c_1=0$. 
 We shall consider two classes of non-K\"ahler Calabi-Yau geometries which we will call Type IIA and Type IIB supersymmetric systems.  In six dimensions, the systems arise directly from the imposition of supersymmetry in Type IIA/IIB string theory  \cite{GMPT, Tomasiello, TY}.   The six-dimensional Type IIA system is a symplectic manifold $(X,\omega)$ with an almost complex structure induced from a complex decomposible three-form $\Omega$ with $\der \mathrm{Re} \, \Omega = 0$ (instead of $\der \Omega = 0$).  Type IIB system is a complex threefold $\check{X}$ with a holomorphic volume form $\check{\Omega}$ and a balanced Hermitian metric, whose associated $(1,1)$ form $\check{\omega}$ satisfies $\der (\check{\omega}^2) = 0$ (instead of $\der \check{\omega} = 0$).  We also define and study their higher dimensional analogs.  Here, we preserve on the IIB side the balanced metric condition which has a number of desirable properties \cite{Michelsohn,AB93,AB95,AB04}.  The mirror dual condition on the symplectic IIA side will be shown to require the introduction of a special real polarization.

Though there is now a fairly decent understanding of mirror symmetry in the K\"ahler situation, the mirror phenomenon for non-K\"ahler Calabi-Yau geometries has not been studied as much and is not well-understood.
This motivates us to study mirror symmetry in the presence of the so-called Ramond-Ramond (RR) flux using SYZ transformation.  The aim is to construct a duality between Type IIA and Type IIB supersymmetric systems (and their higher dimensional analogs) by using SYZ and Fourier-Mukai transform.  

While SYZ fibrations in general may not exist in the presence of flux, the SYZ transformation (when fibations exist) does provide important clues on how mirror symmetry should behave in the non-K\"ahler setting.  We expect mirror symmetry between type IIA and IIB systems still to occur in a similar way even in cases when SYZ fibrations may not exist.  Namely, the flux source terms (denoted as $\rho_A$ and $\rho_B$ later on in this paper) are mirror to each other in the sense of homological mirror symmetry.

Construction of non-K\"ahler Calabi-Yau manifolds has been studied in several recent works \cite{Calabi,STY, GP,GGP,Wu-torsion,FP1,FLY,FP2}.  The survey by Fu \cite{Fu} gave an excellent introduction to this topic.  In this paper we focus on the semi-flat setting \cite{LYZ,Leung}, namely we consider Lagrangian fibrations away from singular fibers.  We will see that interesting correspondences between the Type IIA and IIB systems and their RR fluxes already appear in such a simple setting.  Iwasawa manifold and its higher-dimensional analogs serves as important examples.  The main theorem is:

\begin{theorem}[See Theorem \ref{thm:main2}] \label{thm:main}
Under SYZ and Fourier-Mukai transform, a semi-flat supersymmetric Type IIA $SU(n)$ structure is transformed to a semi-flat supersymmetric Type IIB $SU(n)$ structure, and their fluxes also correspond to each other by Fourier-Mukai transform.
\end{theorem}

The study of non-K\"ahler Calabi-Yau geometries was largely motivated by string theory (see, for example \cite{Strominger,BBDG,BBDGS}).  In physical terms, the presence of fluxes in supersymmetric systems results in geometries that are generically non-K\"ahler.  In Type IIA and IIB string theory, there are mainly two types of fluxes, namely Ramond-Ramond (RR) flux and H-flux (which is also called NS-NS flux).  The settings of Type IIA and IIB supersymmetric systems in this paper follows the form expressed in Tseng-Yau \cite{TY}, in which an RR flux is present.  These are special cases of the more general systems as given in Grana-Minasian-Petrini-Tomasiello \cite{GMPT} and Tomasiello \cite{Tomasiello} which include the H-flux and are expressed in terms of generalized complex geometry \cite{Hitchin-gen,Gualtieri}   In this paper, we place our focus on mirrors of balanced metrics and generalize the definition of Type IIA systems to higher dimensions, which does not involve the use of generalized geometry.

T-duality and SYZ in the setting of generalized Calabi-Yau manifolds were studied in \cite{BEM,BHM,FMT,GS-gen,CG} where H-flux played the key role, and also from the bispinor perspective which includes RR fluxes in \cite{GMPT2, GMPW}.  The emphasis of this paper is different: we study SYZ in the presence of RR-flux and in particular focusing on the balanced Hermitian metric condition on non-K\"ahler manifolds, which also appears as part of the Strominger system \cite{Strominger}.

\section*{Acknowledgement}
We thank A. Tomasiello and C.-J. Tsai for discussions.  We are also grateful to referees for useful comments.  This work is supported in part by NSF grants DMS-1159412, 1308244 and PHY-1306313.

\section{The supersymmetric systems and their higher-dimensional analogs}
In this section we introduce Type IIA and IIB supersymmetric $SU(n)$ systems, which are the main objects of study in this paper.

\begin{defn}
Let $X$ be a real manifold of dimension $2n$.  An $SU(n)$ structure on $X$ is a pair  $(\omega, \Omega)$ of differential forms satisfying the following conditions:
\begin{enumerate}
\item $\Omega$ is a nowhere-vanishing decomposible complex-valued $n$-form on $X$ such that by defining
$$T^{(0,1)}X := \{ v \in TX \otimes \cpx: \iota_v \Omega = 0\}$$
and $T^{(1,0)}X$ to be the complex conjugate of $T^{(0,1)}X$, one has a splitting
$$ TX \otimes \cpx = T^{(1,0)}X \oplus T^{(0,1)}X $$
which induces an almost complex structure $J$ on $X$.  Then $\Omega$ is an $(n,0)$ form with respect to this almost complex structure $J$.  (Note that we do not require $\der \Omega = 0$ here.)  
\item $\omega$ is a non-degenerate real\footnote{More generally we should include $B$-field here, and $\omega$ can be complex-valued.  Such a complexification is necessary for making the mirror map to be an isomorphism.  This paper focuses on RR flux and we omit it for simplicity.} $(1,1)$-form with respect to this complex structure $J$ such that $\omega(\cdot, J \cdot)$ defines a Hermitian metric on $X$.  (Note that we do not require $\der \omega = 0$ here.) This in particular implies that $\Omega \wedge \omega = 0$.
\end{enumerate}

Since both $(\Omega \wedge \bar{\Omega})/\bi^n$ and $\omega^n$ are real nowhere-vanishing top-forms, we have
$$ \Omega \wedge \bar{\Omega} = \consti^n \, F \cdot \frac{\omega^n}{n!} $$
for some nowhere-vanishing function $F$ on $X$.  $F$ is called to be the conformal factor of the $SU(n)$ structure.
\end{defn}

In other words an $SU(n)$ structure is a pointwise Calabi-Yau structure without imposing any integrability condition.  Supersymmetry imposes integrability conditions on $SU(n)$ structures.  In string theory one mostly focuses on $n=3$.

\begin{defn} \label{type IIB n=3}
An $SU(3)$ structure $(X, \omega, \Omega)$ is said to be supersymmetric of type IIB if $\Omega$ defines an honest complex structure and $\omega$ defines a balanced metric, that is, $\der \Omega = 0$ and $\der (\omega^{2}) = 0$.

Let $F$ be the conformal factor of the $SU(3)$ structure.  Define
$$\rho_B = 2 \consti \partial \bar{\partial} \left(F^{-1} \cdot \, \omega \right)$$
which is an exact $(2,2)$ form.
\end{defn}

In conclusion, we have the following system of equations for a supersymmetric $SU(3)$ structure of type IIB:
\begin{align}
\der \Omega &= 0,\\
\der (\omega^2) &= 0,\\
\Omega \wedge \bar{\Omega} &= -\consti \, F \cdot \frac{\omega^3}{6},\\
2 \consti\partial \bar{\partial} \left(F^{-1} \cdot \omega\right) &= \rho_B.
\end{align}
When $\omega$ is symplectic (that is $\der \omega = 0$) and $F$ is constant, $(X,\omega,\Omega)$ is Calabi-Yau, and the RR flux source current $\rho_B$ is zero.  In such case the corresponding metric has $SU(3)$ holonomy.

\begin{defn} \label{Type IIA n=3}
An $SU(3)$ structure $(X, \omega, \Omega)$ is said to be supersymmetric of Type IIA if $\der \omega = 0$ and $\der \mathrm{Re}\, \Omega = 0$.

Let $F$ be the conformal factor of the $SU(3)$ structure.  We define
$$\rho_A = \der\der^{\Lambda} (F \cdot \mathrm{Im}\, \Omega) $$
which is an exact three-form.  Here
$$\der^{\Lambda} = \der \Lambda - \Lambda \der$$
is the symplectic adjoint operator, where $\Lambda$ is the adjoint of the Lefschetz operator $L = \omega \wedge (\cdot)$.
\end{defn}

Thus the Type IIA supersymmetric $SU(3)$ structure satisfies the following system of equations:
\begin{align}
\der \omega &= 0,\\
\der (\mathrm{Re} \, \Omega) &= 0,\\
\Omega \wedge \bar{\Omega} &= -\consti \, F \cdot \frac{\omega^3}{6}, \\
\der\der^{\Lambda} (F \cdot \mathrm{Im}\, \Omega) &= \rho_A.
\end{align}
When $\der \Omega = 0$ and $F$ is constant, $\Omega$ defines an honest complex structure and $X$ is Calabi-Yau; the RR flux source current $\rho_A$ is zero.

Type IIB supersymmetric structures have a natural generalization to all dimensions:

\begin{defn} \label{type-B}
An $SU(n)$ structure $(X, \omega, \Omega)$ is said to be supersymmetric of type IIB if $\Omega$ defines an honest complex structure and $\omega$ defines a balanced metric, that is, $\der \Omega = 0$ and $\der (\omega^{n-1}) = 0$.
\end{defn}

\begin{remark}
It is well known that supersymmetric $SU(3)$ system satisfies the equation of motion.  See, for example \cite{GMPT2}.    
\end{remark}

Nilmanifolds introduced in Section \ref{nilmfd} serve as examples of such systems.

However, generalizing Type IIA supersymmetric $SU(n)$ structure to all dimensions is not as obvious as it appears.  The naive definition, which is to require $\der \mathrm{Re} \Omega =0$ as in dimension three,  does not match with SYZ transformation.  It turns out that choosing a special real polarization is important in the definition of Type IIA supersymmetric structure in higher dimensions.

A complex structure automatically comes with a complex polarization, which is the splitting
$$ TX \otimes \cpx = T^{(1,0)}X \oplus T^{(0,1)}X. $$
This gives rise to the notion of $(p,q)$ forms.  On the other hand, given a symplectic structure, it does not automatically come with a real polarization, and we have to fix one in order to define Type IIA supersymmetric structure in higher dimensions.

\begin{defn}
Let $X$ be a real $2n$-fold and $\omega$ a non-degenerate two-form.  A real polarization with respect to $\omega$ is an integrable Lagrangian distribution $\{\Lambda_x \subset T_x X: x \in U\}$ over an open dense subset $U \subset X$.  (Lagrangian means that $\omega|_{\Lambda_x} = 0$ for all $x \in U$.)

Let $(X,\omega,\Omega)$ be an $SU(n)$ structure.  A real polarization $(U,\Lambda)$ is said to be special with respect to $\Omega$ of phase $\theta \in \R / 2\pi \Z$ if $\Omega|_{\Lambda_x} \in \R_{>0} \cdot \conste^{\consti \theta}$ for all $x \in U$.
\end{defn}

Note that the symbol $\Lambda$ appearing in $\der^\Lambda$ refers to the adjoint of the Lefzchetz operator.  We will also use $\Lambda$ to denote a real polarization.  It should be clear from the context which meaning we refer to.

Given a real polarization on $(X,\omega)$, the metric on $X$ induces a splitting $TU = \Lambda \oplus \Lambda^\perp$.  We have
$$ \Omega^n_{U} = \bigoplus_{p+q=n} \Omega_U^{(p,q)^\Lambda}, $$
where forms in $\Omega_U^{(p,q)^\Lambda}$ are defined over $U$ and consist of $p$ $\Lambda$-directions and $q$ $\Lambda^\perp$-directions, and the corresponding projection operators $\pi_{\Lambda}^{p,q}:\Omega^n_{U} \to \Omega_U^{(p,q)^\Lambda}$.  Let $\Omega^{(p,q)^\Lambda}$ be the space of differential forms on $X$ whose restriction to $U$ belongs to $\Omega_U^{(p,q)^\Lambda}$.

Then define a generalization of Definition \ref{Type IIA n=3} as follows:

\begin{defn} \label{Type IIA}
Let $(X,\omega,\Omega)$ be an $SU(n)$ structure with a special real polarization $(U,\Lambda)$.  $(X, \omega, \Omega)$ is said to be supersymmetric of Type IIA if
\begin{enumerate}
\item $\der \omega = 0$.
\item $\der (\pi_{\Lambda}^{n,0} \cdot \Omega|_U) = 0$.
\item $\der (\pi_{\Lambda}^{1,n-1} \cdot \Omega|_U) = 0$.
\end{enumerate}
\end{defn}

\begin{defn} \label{def:flux}
Let $F$ be the conformal factor of an $SU(n)$ structure $(X,\omega,\Omega)$.  If $(X,\omega,\Omega)$ is supersymmetric of Type IIA (with respect to a special real polarization $(U,\Lambda)$), define
$$\rho_A = -\consti \der\der^{\Lambda} (F\cdot (\pi_{\Lambda}^{n-1,1} \cdot \Omega|_U + \pi_{\Lambda}^{0,n} \cdot \Omega|_U)).$$
If $(X,\omega,\Omega)$ is supersymmetric of Type IIB, define 
$$ \rho_B = 2\consti \partial \bar{\partial} \left(F^{-1} \cdot \omega \right).$$
\end{defn}

Definition \ref{Type IIA} reduces to Definition \ref{Type IIA n=3} when $n=3$:

\begin{prop}
Let $(X,\omega,\Omega)$ be an $SU(3)$ structure with a special real polarization $(U,\Lambda)$ of phase zero or $\pi$.  Then $\der (\mathrm{Re} \, \Omega) = 0$ if and only if $\der (\pi_{\Lambda}^{3,0} \cdot \Omega|_U) = \der (\pi_{\Lambda}^{1,2} \cdot \Omega|_U) = 0$.
\end{prop}

\begin{proof}
Since $U$ is an open dense subset of $X$, $\der (\mathrm{Re} \, \Omega) = 0$ if and only if $\der (\mathrm{Re} \, \Omega|_U) = 0$.  For $x_0 \in U$, let $(\theta_1,\theta_2,\theta_3,r_1,r_2,r_3)$ be local coordinates around $x$ such that each leaf of $\Lambda$ is given by $(r_1,r_2,r_3) = (c_1,c_2,c_3)$ for some constants $c_i$'s.  Then $\Omega$ can be written as
$$ \Omega = f(\theta,r) \cdot (\der\theta_1 + \consti \psi_1) \wedge (\der\theta_2 + \consti \psi_2) \wedge (\der\theta_3 + \consti \psi_3) $$
where $\psi_i = J \cdot \der\theta_i$ are one-forms and $f(\theta,r)$ is complex-valued.  Since $(U,\Lambda)$ is special of phase zero or $\pi$,  $\Omega|_{\Lambda_x} = f(\theta,r) \cdot \der\theta_1 \wedge \der\theta_2 \wedge \der\theta_3$ is real-valued, and hence $f(\theta,r)$ is real-valued.  Then
\begin{align*}
\mathrm{Re} \, \Omega &= f(\theta,r) \cdot (\der\theta_1 \wedge \der\theta_2 \wedge \der\theta_3 - \der\theta_1 \wedge \psi_2 \wedge \psi_3 - \psi_1 \wedge \der\theta_2 \wedge \psi_3 - \psi_1\wedge\psi_2\wedge\der\theta_3)\\
&= (\pi_{\Lambda}^{3,0} \oplus \pi_{\Lambda}^{1,2}) \cdot \Omega.
\end{align*}
It follows that $\der (\mathrm{Re} \, \Omega) = 0$ if and only if $\der (\pi_{\Lambda}^{3,0} \cdot \Omega|_U) = \der (\pi_{\Lambda}^{1,2} \cdot \Omega|_U) = 0$.
\end{proof}

Note also when $n=3$, $\mathrm{Im} \, \Omega  = -\consti (\pi_{\Lambda}^{2,1} \cdot \Omega + \pi_{\Lambda}^{0,3} \cdot \Omega)$.  Therefore, to conclude, a Type IIA supersymmetric $SU(n)$ system satisfies
\begin{align}
\label{sympl} \der \omega &= 0,\\
\label{balanced B} \der ((\pi_{\Lambda}^{n,0} \oplus \pi_{\Lambda}^{1,n-1})\cdot \Omega) &= 0,\\
\label{oOA} \Omega \wedge \bar{\Omega} &= -\consti^n \, F \cdot \frac{\omega^n}{n!}, \\
\label{flux B} -\consti\der\der^{\Lambda} (F\cdot(\pi_{\Lambda}^{n-1,1} \cdot \Omega + \pi_{\Lambda}^{0,n} \cdot \Omega)) &= \rho_A.
\end{align}
A Type IIB supersymmetric $SU(n)$ system satisfies
\begin{align}
\label{holo} \der \Omega &= 0,\\
\label{balanced} \der (\omega^{n-1}) &= 0,\\
\label{oOB} \Omega \wedge \bar{\Omega} &= -\consti \, F \cdot \frac{\omega^n}{n!},\\
\label{flux} 2\consti \partial \bar{\partial} \left(F^{-1} \cdot \omega\right) &= \rho_B.
\end{align}
\begin{remark}
The general $SU(n)$ systems we introduced here have similarities with the higher dimensional supersymmetric compactifications in string theory.  For instance, for n=4, they are very similar to special cases of the supersymmetric equations in \cite{PT, Rosa}.  It is interesting to expand upon this issue in future works.   
\end{remark}

\section{SYZ transformation} \label{sf}
The SYZ transform we use in this paper follows from \cite{LYZ,Leung}.  In this section we briefly review the setting and definitions which will be used in the rest of the paper.

Let $(X,\omega)$ be a symplectic manifold and $\pi: X \to B$ a Lagrangian torus bundle.  By Arnold-Liouville's action-angle coordinates \cite[Section 50]{arnold_book}, for each $p \in B$ there exists open subset $U \subset B$ containing $p$, action coordinates $r_1, \ldots, r_n$ of $U$ and symplectomorphism $(\pi^{-1} (U), \omega) \cong (T^*U/\Lambda^*, \omega_{\mathrm{can}})$, where $\Lambda^*$ is a lattice bundle generated by $\der r_1, \ldots, \der r_n$.  The corresponding fiber coordinates of $T^*U$ are denoted as $\theta_i$'s, and so $\omega = \sum_{i=1}^n \der \theta_i \wedge \der r_i$.

The dual torus bundle is defined as the set of all fiberwise flat $U(1)$-connections:
$$\check{X} := \{(r, \conn): r \in B, \conn \textrm{ is a flat } U(1) \textrm{ connection over } F_r \}$$
where $F_r$ denotes the fiber of $\pi$ at $r$.  The torus bundle map $\check{\pi}: \check{X} \to B$ is given by forgetting the fiberwise flat $U(1)$ connections $\conn$.

$\check{X}$ is endowed with a canonical complex structure.  Locally for each $p \in B$, there exists open subset $U \subset B$ containing $p$ and a biholomorphism $\check{\pi}^{-1} (U) \cong TU/\Lambda$, where $\Lambda \subset TU$ is the lattice bundle generated by $\dd{r_1}, \ldots, \dd{r_n}$, and $r_i$'s are the action coordinates mentioned before.  The corresponding fiber coordinates of $TU$ are denoted as $\check{\theta}_i$'s, and so complex coordinates of $\check{X}$ can be taken to be $\zeta_i := \exp (\check{\theta}_i + \consti r_i)$ or $z_i := \check{\theta}_i + \consti r_i$ for $i=1,\ldots,n$.

The transition between any two local action-coordinate systems of $B$ belongs to $GL(n,\integer) \ltimes \real^n$,  and so $B$ is endowed with a tropical affine structure.  We assume that the action coordinate systems can be chosen such that all the transitions belong to $SL(n,\integer) \ltimes \real^n$.  In such a situation $\check{X}$ has a holomorphic volume form $\check{\Omega}$, which is locally written as $$\check{\Omega} = \der z_1 \wedge \ldots \wedge \der z_n = (\der\check{\theta}_1 + \consti \der r_1) \wedge \ldots \wedge (\der\check{\theta}_n + \consti \der r_n).$$
$(X,\omega)$ and $(\check{X}, \check{\Omega})$ are said to form a semi-flat mirror pair.  The semi-flat $A$-branes on $X$ are mirror to semi-flat $B$-branes on $\check{X}$ under Fourier-Mukai transform.  The readers are referred to \cite{LYZ} for details.

In this paper, we mostly focus on Fourier-Mukai transform of differential forms, which will be discussed in the next section.  We will consider a semi-flat real $(1,1)$ form $\check{\omega}$ on $\check{X}$, which is given by
$$ \check{\omega} = \frac{\consti}{2} \sum_{i,j} \mu_{ij}(r)\, \der z_i \wedge \overline{\der z_j} $$
where $\sum_{i,j} \mu_{ij}(r) \der r_i \der r_j$ defines a Riemannian metric on $B$ (that is, $\mu_{ij}$ is symmetric on $i,j$ and positive definite).  We will impose the balanced condition $\der \check{\omega}^{n-1} = 0$ instead of the K\"ahler condition $\der \check{\omega} = 0$, and consider its mirror structure by applying the Fourier-Mukai transform.

\section{Fourier-Mukai transform of differential forms} \label{Fourier}
Let $(X, \omega)$ and $(\check{X}, \check{\Omega})$ be a semi-flat mirror pair of real dimension $2n$ as in Section \ref{sf}.  In this section we introduce Fourier-Mukai transform for differential forms on $X$ and $\check{X}$.  We will restrict the construction to $T$-invariant differential forms, while keeping in mind that Fourier-Mukai transform can also be done for general differential forms, which is an important tool for quantum corrections \cite{CL,CLL}.  Fourier-Mukai transform of differential forms was used to study SYZ mirror symmetry without flux in \cite{Leung,CL,CLL,CLM}.  There is an important subtlety in the treatment here, namely we take a switch from the complex to the real polarization (Definition \ref{switch}) along with the transform.  A new result in this section is a direct relation between the Doubealt operators $(\bar{\partial},\partial)$ on $\check{X}$ and the differentials $(\der,\der^{\Lambda})$ on $X$.

Let $\Omega_B^k (X,\cpx)$ denote the space of complex-valued $k$-forms on $X$ which depend only on base, that is, $\phi \in \Omega_B^k (X,\cpx)$ is locally written as
$$\phi = \sum_{\substack{I=(i_1,\ldots,i_p)\\J=(j_1,\ldots,j_q)\\p+q=k}} a_{IJ}(r) \,\der \theta_{i_1} \wedge \ldots \wedge \der \theta_{i_p} \wedge \der r_{j_1} \wedge \ldots \wedge \der r_{j_q}$$
where $(r_1,\ldots,r_n,\theta_1,\ldots,\theta_n)$ is an action-angle coordinate system of $X$, and $a_{IJ}(r)$'s are complex-valued functions on $B$.  Such a $\phi$ is also said to be $T$-invariant, where $T$ denotes a torus fiber of $X \to B$.

\begin{defn}
A semi-flat (supersymmetric) $SU(n)$ structure is a (supersymmetric) $SU(n)$ structure $(X,\omega,\Omega)$ such that $X$ has a Lagrangian torus bundle structure $X \to B$ and $\omega,\Omega \in \Omega^*_B(X,\C)$.
\end{defn}

Similarly, let $\Omega_B^{p,q}(\check{X})$ denote the space of $(p,q)$-forms on $\check{X}$ which depend only on base.  $\check{\phi} \in \Omega_B^{p,q}(\check{X})$ is locally of the form
$$\check{\phi} = \sum_{\substack{I=(i_1,\ldots,i_p)\\J=(j_1,\ldots,j_q)}} a_{IJ}(r) \,\der z_{i_1} \wedge \ldots \wedge \der z_{i_p} \wedge \overline{\der z_{j_1}} \wedge \ldots \wedge \overline{\der z_{j_q}}.$$

The above expression of $\check{\phi}$ is written in terms of complex polarization.  We can switch to real polarization by the following definition:

\begin{defn}[Polarization switch operator] \label{switch}
The polarization switch operator $\cP$ on $\Omega_B^*(\check{X})$ is defined by sending
$$\check{\phi} = \sum_{\substack{I=(i_1,\ldots,i_p)\\J=(j_1,\ldots,j_q)}} a_{IJ}(r) \,\der z_{i_1} \wedge \ldots \wedge \der z_{i_p} \wedge \overline{\der z_{j_1}} \wedge \ldots \wedge \overline{\der z_{j_q}}$$
to
$$\cP \cdot \check{\phi} = \sum_{\substack{I=(i_1,\ldots,i_p)\\J=(j_1,\ldots,j_q)}} a_{IJ}(r) \,\der \check{\theta}_{i_1} \wedge \ldots \wedge \der \check{\theta}_{i_p} \wedge \der r_{j_1} \wedge \ldots \wedge \der r_{j_q}.$$
\end{defn}

Now we define Fourier-Mukai transform of differential forms on $X$ and $\check{X}$, which gives an isomorphism
$$ \Omega_B^k (X,\cpx) \cong \bigoplus_{p = 0}^k \left(\Omega_B^{p,n-k+p} (\check{X})\right). $$
 
Consider the following commutative diagram:
$$
\begin{diagram} \label{FM}
	\node[2]{X \times_{B} \check{X}} \arrow{sw} \arrow{se} \\
	\node{X} \arrow{se,b}{\pi} \node[2]{\check{X}} \arrow{sw,b}{\check{\pi}} \\
	\node[2]{B}
\end{diagram}
$$
where $\pi: X \to B$ and $\check{\pi}: \check{X} \to B$ are dual torus bundle maps given in the semi-flat setting, and the fiber product $X \times_{B} \check{X}$ is defined using these bundle maps.  We have the universal connection on $X \times_{B} \check{X}$, which is $\der + \bi \check{\theta}_i \der\theta_i - \bi \theta_i \der \check{\theta}_i$ in terms of semi-flat local coordinates.  The curvature of the universal connection is
$$ \bF = 2\bi \sum_{i=1}^n \der\check{\theta}_i \wedge \der \theta_i. $$

\begin{defn}[Fourier-Mukai transform] \label{defn:FT}
Let $\check{\phi}$ be a $T$-invariant differential form on $\check{X}$.
The Fourier transform of $\check{\phi}$ is defined as
$$ \FT \cdot \check{\phi} := \pi_* \left((\check{\pi}^*(\cP \cdot \check{\phi})) \wedge \exp \frac{\bF}{2\bi}\right) $$
where the push-forward $\pi_*$ induced by $\pi:X \to B$ is defined by integration along torus fibers.

Similarly for a $T$-invariant differential form $\phi$ on $X$, its Fourier transform is defined as
$$ \FT \cdot \phi := \cP^{-1} \cdot \left(\check{\pi}_* \left((\pi^*\phi)\wedge \exp \frac{-\bF}{2\bi}\right)\right). $$
\end{defn}

By carefully keeping track of the sign in the transform, we see that applying Fourier-Mukai transform two times is simply identity up to a sign only dependent on the dimension.  

\begin{prop}
$$ \FT \circ \FT = (-1)^{\frac{n(n-1)}{2}}$$
where $n = \dim B$.
\end{prop}

\begin{proof}
It suffices to consider the basic elements $\der z_I \wedge \overline{\der z_J} = \der z_{i_1} \wedge \ldots \wedge \der z_{i_p} \wedge \overline{\der z_{j_1}} \wedge \ldots \wedge \overline{\der z_{j_q}}$.  Switching to real polarization gives $\der \check{\theta}_I \wedge \der r_J = \der \check{\theta}_{i_1} \wedge \ldots \wedge \der \check{\theta}_{i_p} \wedge \der r_{j_1} \wedge \ldots \wedge \der r_{j_q}$.  Multiplication by
$$ \exp \left( \sum_i \der \check{\theta}_i \wedge \der \theta_i \right) = \sum_{p=0}^n (-1)^{\frac{p(p-1)}{2}} \sum_{K = \{i_1,\ldots,i_p\}} \der \check{\theta}_K \wedge \der \theta_K$$
gives
$$ (-1)^{\frac{(n-p)(n-p-1)}{2}} \der \check{\theta}_{I^c} \wedge \der \theta_{I^c} \wedge \der \check{\theta}_I \wedge \der r_J = (-1)^{\frac{(n-p)(n-p-1)}{2}} (-1)^{p(n-p)} \der \check{\theta}_{I^c} \wedge \der \check{\theta}_I \wedge \der \theta_{I^c} \wedge \der r_J$$
where $I^c := \{1,\ldots,n\} - I$ and $p = |I|$.
Then integrating along the fiber directions $\check{\theta}_i$'s gives
$$ (-1)^{\frac{(n-p)(n-p-1)}{2}} (-1)^{p(n-p)} \sign(I^c,I) \der \theta_{I^c} \wedge \der r_J = (-1)^{\frac{(n-p)(n-p-1)}{2}} \sign(I,I^c) \der \theta_{I^c} \wedge \der r_J.$$
Thus we have
\begin{equation} \label{eq:FT}
\FT \cdot (\der z_I \wedge \overline{\der z_J}) = (-1)^{\frac{(n-p)(n-p-1)}{2}} \sign(I,I^c) \der \theta_{I^c} \wedge \der r_J.
\end{equation}

Now we take Fourier transform again.  Multiplying $\der \theta_{I^c} \wedge \der r_J$ by
$\exp \left( - \sum_i \der \check{\theta}_i \wedge \der \theta_i \right)$ gives
gives
$$ (-1)^p (-1)^{\frac{p(p-1)}{2}} \der \check{\theta}_I \wedge \der \theta_I \wedge \der \theta_{I^c} \wedge \der r_J = (-1)^p (-1)^{\frac{p(p-1)}{2}} (-1)^{np} \der \theta_I \wedge \der \theta_{I^c} \wedge \der \check{\theta}_I \wedge \der r_J.$$
Integrating along the fiber directions $\theta_i$'s gives
$$ (-1)^p (-1)^{\frac{p(p-1)}{2}} (-1)^{np} \sign(I,I^c) \der \check{\theta}_I \wedge \der r_J.$$
Switching the real polarization back to complex polarization, we have
$$ \FT (\der \theta_{I^c} \wedge \der r_J) = (-1)^p (-1)^{\frac{p(p-1)}{2}} (-1)^{np} \sign(I,I^c) \der z_I \wedge \der \overline{z_J}.  $$
Hence $\FT \circ \FT (\der z_I \wedge \overline{\der z_J})$ equals to
$$ (-1)^{\frac{(n-p)(n-p-1)}{2}} (-1)^p (-1)^{\frac{p(p-1)}{2}} (-1)^{np} \cdot \der z_I \wedge \overline{\der z_J} = (-1)^{\frac{n(n-1)}{2}} \der z_I \wedge \overline{\der z_J}.$$
\end{proof}

In the theorem below we see that $\der$ on $\Omega_B^* (X,\cpx)$ is transformed to $\bar{\partial}$ on $\Omega_B^* (\check{X},\cpx)$, and $\der^{\Lambda}$ on $\Omega_B^*(X,\cpx)$ is transformed to $\partial$ on $\Omega_B^* (\check{X},\cpx)$.  We recall that $\der^\Lambda := \der \circ \Lambda - \Lambda \circ \der$ which decreases the degree by one, where $\Lambda$ is the dual Lefschetz operator.  To our knowledge the theorem has not appeared in previous literature.

\begin{theorem} \label{thm:FT}
Fourier-Mukai transform gives an isomorphism
$$ \left(\Omega_B^* (X,\cpx), \frac{(-1)^n\consti}{2} \der, \frac{(-1)^n\consti}{2} \der^{\Lambda}\right) \cong \left(\Omega_B^* (\check{X},\cpx), \bar{\partial}, \partial\right) $$
as double complexes.
\end{theorem}

\begin{proof}
Let $\phi = \phi_{I,J} \der z_I \wedge \der z_J$ where $I$ and $J$ are multi-indices.  By assuming $\phi_{I,J}$ is skew-symmetric on $I$ and $J$, the expression is independent of the orderings of $I$ and $J$.  We need to prove that
$$ \FT \circ \bar{\partial} \cdot \phi = \frac{(-1)^n\consti}{2} \cdot \der \circ \FT \cdot \phi $$
and
$$ \FT \circ \partial \cdot \phi = \frac{(-1)^n\consti}{2} \cdot \der^{\Lambda} \circ \FT \cdot \phi.$$

Consider the first equation.  By Equation \eqref{eq:FT}, Fourier transform of $\phi$ is
\begin{equation} \label{eq:FT-phi}
\FT \cdot \phi = (-1)^{\frac{(n-p)(n-p-1)}{2}} \sign(I,I^c) \phi_{I,J} \der \theta_{I^c} \wedge \der r_J
\end{equation}
where $I^c = \{1,\ldots,n\} - I$.
Note that the above expression is independent of the orderings of $I^c$ and $J$.

Let $p := |I|$.  We have
\begin{align*}
\FT \circ \bar{\partial} \cdot \phi &= \FT \cdot \left( \frac{\consti}{2} (\partial_i\phi_{I,J}) \overline{\der z_i} \wedge \der z_I \wedge \overline{\der z_J} \right) \\
&= \frac{\consti}{2} (-1)^{\frac{(n-p)(n-p-1)}{2}} \sign(I,I^c) (\partial_i\phi_{I,J}) (-1)^p  \der \theta_{I^c} \wedge \der r_i \wedge \der r_J  \\
&= \frac{\consti}{2} (-1)^{\frac{(n-p)(n-p-1)}{2}} \sign(I,I^c) (\partial_i\phi_{I,J}) (-1)^n  \der r_i \wedge \der \theta_{I^c} \wedge \der r_J.
\end{align*}
Comparing with Equation \ref{eq:FT-phi}, we see that $\FT \circ \bar{\partial} \cdot \phi = (-1)^n \frac{\consti}{2} \cdot \der \circ \FT \cdot \phi$.

To prove the second equation, let's recall that $\der^\Lambda = \der \circ \Lambda - \Lambda \circ \der$, where
$$ \Lambda \phi = \sum_{i=1}^n \iota_{r_i} \iota_{\theta_i} \phi $$
for $\omega = \sum_i \der \theta_i \wedge \der r_i$.  We remark that more generally for $\omega = \sum_{i,j} \omega_{ij} \der x_i \wedge \der x_j$,
$$\Lambda \phi = \frac{1}{2} \sum_{i,j} (\omega^{-1})^{ij} \iota_{x_i} \iota_{x_j} \phi.$$

Writing $K = I^c - \{i\}$, we have
\begin{align*}
\FT \circ \partial \cdot \phi &= \FT \cdot \left( \frac{-\consti}{2} \sum_{i \in I^c} \partial_i \phi_{I,J} \, \der z_i \wedge \der z_I \wedge \overline{\der z_J} \right) \\
&= \frac{-\consti}{2} (-1)^{\frac{(n-p-1)(n-p-2)}{2}} \sum_{i \in I^c} \sign(i,I,K) (\partial_i \phi_{I,J}) \, \der \theta_K \wedge \der r_J.
\end{align*}

On the other hand, let's compute $\der \circ \Lambda \circ \FT \cdot \phi$ and $\Lambda \circ \der \circ \FT \cdot \phi$. 
$$ \der \circ \FT \cdot \phi = (-1)^{\frac{(n-p)(n-p-1)}{2}} \sign(I,I^c) \sum_{i \not\in J} (\partial_i \phi_{I,J}) \, \der r_i \wedge \der \theta_{I^c} \wedge \der r_J.  $$
Again the above expression is independent of orderings of $I^c$ and $J$.  Since $\Lambda = \sum_j \iota_{r_j} \iota_{\theta_j}$, taking $I^c = (j,K)$ and $J=(j,L)$, we have
\begin{align*}
\Lambda \circ \der \circ \FT \cdot \phi =&(-1)^{\frac{(n-p)(n-p-1)}{2}} \sign(I,I^c) \sum_{\substack{j=i \not\in J\\j \in I^c}} (\partial_i \phi_{I,J}) \, (- \der \theta_{K} \wedge \der r_J) \\
&+(-1)^{\frac{(n-p)(n-p-1)}{2}} \sign(I,I^c) \sum_{\substack{i \not\in J\\j \not=i \\j \in I^c \cap J}} (\partial_i \phi_{I,J}) \, \der r_i \wedge ((-1)^{n-p-1} \der \theta_K \wedge \der r_L).
\end{align*}

To compute $\der \circ \Lambda \circ \FT \cdot \phi$ (again taking $I^c = (j,K)$ and $J=(j,L)$), we have
$$ \Lambda \circ \FT \cdot \phi = (-1)^{\frac{(n-p)(n-p-1)}{2}} \sign(I,I^c) \sum_{j \in I^c \cap J} \phi_{I,J} (-1)^{n-p-1} \der \theta_{K} \wedge \der r_L$$
and hence
\begin{align*}
\der \circ \Lambda \circ \FT \cdot \phi =& (-1)^{\frac{(n-p-2)(n-p-1)}{2}} \sign(I,I^c) \sum_{i=j \in I^c \cap J} \partial_j \phi_{I,J} \, \der r_j \wedge \der \theta_{K} \wedge \der r_L \\
&+ (-1)^{\frac{(n-p)(n-p-1)}{2}} \sign(I,I^c) \sum_{\substack{i \not\in J\\i\not=j\\ j \in I^c \cap J}} \partial_i \phi_{I,J} \, \der r_i \wedge (-1)^{n-p-1} \der \theta_{K} \wedge \der r_L.
\end{align*}
We see that the second terms of $\Lambda \circ \der \circ \FT \cdot \phi$ and $\der \circ \Lambda \circ \FT \cdot \phi$ are equal.  Thus
\begin{align*}
\der^\Lambda \circ \FT \cdot \phi =& (-1)^{\frac{(n-p-2)(n-p-1)}{2}} \sign(I,I^c) \sum_{j \in I^c \cap J} \partial_j \phi_{I,J} \, \der r_j \wedge \der \theta_{K} \wedge \der r_L \\
& - (-1)^{\frac{(n-p)(n-p-1)}{2}} \sign(I,I^c) \sum_{\substack{i \not\in J\\j \in I^c}} (\partial_i \phi_{I,J}) \, (- \der \theta_{K} \wedge \der r_J) \\
=& (-1)^{\frac{(n-p)(n-p-1)}{2}} \sum_{i \in I^c} \sign(I,I^c) (\partial_i \phi_{I,J}) \, \der \theta_{K} \wedge \der r_J \\
=& (-1)^{n-1} (-1)^{\frac{(n-p-2)(n-p-1)}{2}} \sum_{i \in I^c} \sign(i,I,K) (\partial_i \phi_{I,J}) \, \der \theta_{K} \wedge \der r_J \\
=& (-1)^n \left(\frac{\consti}{2}\right)^{-1} \FT \cdot (\partial \phi).
\end{align*}
\end{proof}

\section{SYZ between type-A and type-B supersymmetric systems} \label{SYZ}

Let $(X, \omega)$ and $(\check{X}, \check{\Omega})$ be a semi-flat mirror pair as described in Section 
\ref{sf}.  $(X, \omega)$ has the required real polarization $\Lambda$ given by the given Lagrangian torus bundle (where $\Lambda_x$ is the tangent space of the torus fiber at $x\in X$).  The purpose of this section is to prove Theorem \ref{thm:main}, namely Fourier-Mukai transform introduced in Section \ref{Fourier} gives a symmetry between Type-A and Type-B supersymmetric $SU(n)$ structures.  

\begin{theorem}[Detailed version of Theorem \ref{thm:main}] \label{thm:main2}
Let $\check{\omega}$ be a torus-invariant real $(1,1)$-form on $\check{X}$ and $\Omega$ the Fourier-Mukai transform of $\conste^{2 \check{\omega}}$.

\begin{enumerate}
\item $(\check{X},\check{\omega},\check{\Omega})$ forms an $SU(n)$ structure if and only if $(X, \omega, \Omega)$ forms an $SU(n)$ structure.  For such a pair the conformal factor $F$ of $(X, \omega, \Omega)$ is related to the conformal factor $\check{F}$ of $(\check{X}, \check{\omega}, \check{\Omega})$ by $\check{F} = 2^{2n} F^{-1}.$
\item $(X, \omega, \Omega)$ is supersymmetric of Type-A if and only if $(\check{X}, \check{\omega}, \check{\Omega})$ is supersymmetric of Type-B.
\item   Let $\rho_A$ be the Type-A RR flux source current of $(X, \omega, \Omega)$.  Then its Fourier-Mukai transform equals to the Type-B RR flux source current $\check{\rho}_B$ of $(\check{X}, \check{\omega}, \check{\Omega})$ up to a constant multiple.
\end{enumerate}
\end{theorem}

\begin{proof}
Let $(\theta_1,\ldots,\theta_n,r_1,\ldots,r_n)$ be local semi-flat coordinates.  The $(1,1)$ form $\check{\omega}$ can be written as
\begin{align*}
\check{\omega} &= \frac{\consti}{2} \sum_{i,j} \mu_{ij}(r) \der z_i \wedge \overline{\der z_j}\\
&= \sum_{i,j} \mu_{ij}(r) \der \check{\theta}_i \wedge \der r_j + \frac{\consti}{2} \left( \sum_{i<j} (\mu_{ij}-\mu_{ji})(\der\check{\theta}_i \wedge \der\check{\theta}_j + \der r_i \wedge \der r_j) \right).
\end{align*}
Since $\check{\omega}$ is real, $\mu_{ij} = \mu_{ji}$ are real-valued.  Defining the one-forms 
$$\mu_i := \sum_j \mu_{ij} \der r_j$$
on $B$, $\check{\omega}$ can also be written as
$$\check{\omega} = \sum_i \der \check{\theta}_i \wedge \mu_i.$$
Changing to real polarization,
$$ \cP \cdot \exp (2 \check{\omega}) = \exp \left({\consti} \sum_{i,j} \mu_{ij}(r) \der \check{\theta}_i \wedge \der r_j \right).$$
Then
\begin{equation}
\begin{aligned}
\Omega = \FT \cdot \exp (2 \check{\omega}) &= \pi_* \cdot \exp \left( - \sum_i \der \check{\theta}_i \wedge \der \theta_i \right) \wedge \exp \left({\consti} \sum_{i,j} \mu_{ij}(r) \der \check{\theta}_i \wedge \der r_j \right) \\
&= \pi_* \cdot \exp \left( - \sum_i \der \check{\theta}_i \wedge (\der \theta_i + \consti \mu_i) \right) \\
&=  (-1)^{\frac{n(n-1)}{2}} (\der \theta_1 + \consti \mu_1) \wedge \ldots \wedge (\der \theta_n + \consti \mu_n).
\end{aligned} \label{eq:Omega}
\end{equation}

$\Omega$ defines a complex structure $J$ on $X$ by
$$ J \cdot \der \theta_i = -\mu_i, \, J \cdot \mu_i = \der \theta_i$$
if and only if the one-forms $\mu_i$ for $i=1,\ldots,n$ are linearly independent, which is equivalent to the condition that $\omega$ is non-degenerate.

Moreover
$$\omega = \sum_{ij} \der \theta_i \wedge \der r_i = \sum_{ij} \frac{\eta_i + \overline{\eta_i}}{2} \wedge \frac{\mu^{-1}_{ji} (\eta_j - \overline{\eta_j})}{2\consti} = \frac{\consti}{2} \sum_{ij} \mu^{-1}_{ji} \eta_i \wedge \overline{\eta_j}$$
has no $(2,0)$ and $(0,2)$ components because $\mu^{-1}_{ji}$ is symmetric in $i,j$, where $\mu^{-1}$ is the inverse matrix of $\mu_{ij}$ and $\eta_i := \der \theta_i + \consti \mu_i$ are $(1,0)$-forms with respect to $\Omega$.  Thus $\omega$ is a real $(1,1)$ form with respect to $J$.

Finally
\begin{align*}
&(\omega, J) \textrm{ is Hermitian} \Leftrightarrow \mu_{ij} \textrm{ is positive definite}\\
\Leftrightarrow &(\mu^{-1})_{ij} \textrm{ is positive definite} \Leftrightarrow (\check{\omega}, \check{J}) \textrm{ is Hermitian}.
\end{align*}
This proves $(\check{X},\check{\omega},\check{\Omega})$ forms an $SU(n)$ structure if and only if $(X, \omega, \Omega)$ forms an $SU(n)$ structure.

Recall that the conformal factor $F$ for $(X,\omega,\Omega)$ is defined by
$$ \Omega \wedge \bar{\Omega} = \consti^n \cdot F \cdot \frac{\omega^n}{n!} $$
and similarly for the conformal factor $\check{F}$ of $(\check{X},\check{\omega},\check{\Omega})$.  By direct computations, we have
$$ \check{F} = \frac{2^{2n}}{F} = \frac{2^n(-1)^{\frac{n(n+1)}{2}}}{\det(\mu_{ij})}. $$

Now we consider the supersymmetry conditions.  $\der \check{\omega} = 0$ and $\der (\pi_{\Lambda}^{n,0} \cdot \Omega|_U) = 0$ are automatic.  It suffices to prove that $\der (\pi_{\Lambda}^{1,n-1} \cdot \Omega|_U) = 0$ if and only if $\der \check{\omega}^{n-1} = 0$.  The $k$-th summand of
$$ \cP \cdot \exp 2\check{\omega} = \sum_{k=0}^n \frac{(2 \cP \cdot \check{\omega})^k}{k!} $$
has $k$ torus-fiber directions ($\der \check{\theta}$) and $k$ base directions ($\der r$).  After Fourier-Mukai transform it has $n-k$ torus-fiber directions ($\der \theta$) and $k$ base directions ($\der r$).  Thus $\pi_{\Lambda}^{n-k,k} \cdot \Omega$ is the Fourier-Mukai transform of $(2\check{\omega})^{k}/(k)!$ (up to a constant multiple).  By Theorem \ref{thm:FT}, $\der$ is transformed to $\dbar$ (up to a constant multiple).  Thus $\der (\pi_{\Lambda}^{1,n-1} \cdot \Omega)=0$ if and only if $\dbar (\check{\omega}^{n-1}) = 0$.  Since $\check{\omega}$ is real, $\partial(\check{\omega}^{n-1}) = \overline{\dbar (\check{\omega}^{n-1})}$.  Thus $\der (\pi_{\Lambda}^{1,n-1} \cdot \Omega) = 0$ if and only if $\der \check{\omega}^{n-1} = 0$.

Now consider the fluxes of the supersymmetric systems, which are defined by
\begin{align*}
\rho_A &= -\consti \der\der^{\Lambda} (F \cdot (\pi_{\Lambda}^{n-1,1} \cdot \Omega + \pi_{\Lambda}^{0,n} \cdot \Omega)),\\
\check{\rho}_B &= 2\consti \partial \bar{\partial} \left(\check{F}^{-1} \cdot \check{\omega}\right).
\end{align*}

From above $\pi_{\Lambda}^{n-1,1} \cdot \Omega$ is the Fourier-Mukai transform of $2\check{\omega}$, and $\check{F} = 2^{2n} F^{-1}$. The component $\pi_{\Lambda}^{0,n} \cdot \Omega$ is the transform of $\check{\omega}^n$, which contributes zero to $\check{\rho}_B$.  Moreover by Theorem \ref{thm:FT}, $\frac{(-1)^n \consti}{2} \cdot \der$ is transformed to $\dbar$ and $\frac{(-1)^n \consti}{2} \cdot \der^\Lambda$ is transformed to $\partial$.  It follows that the Fourier transform of $\rho_A$ is $2^{2n+2}\check{\rho}_B$.
\end{proof}

\section{Deformation theory}

Since Type IIA and IIB supersymmetric $SU(n)$ systems correspond to each other under the SYZ mirror transform,  their deformations should also correspond to each other.  The following cohomologies are keys to study their deformations:

\begin{defn}[\cite{Bott-Chern,Aeppli}]
Let $\check{X}$ be a complex manifold.  The Bott-Chern cohomology of $\check{X}$ is defined as
$$ H_{\textrm{B.C.}}^{p,q} (\check{X}) := \frac{\Ker (\der) \cap \Omega^{p,q}(\check{X})}{\Image (\partial\bar{\partial}) \cap \Omega^{p,q}(\check{X})}$$
\end{defn}

\begin{defn}[\cite{Tseng-Yau1,Tseng-Yau2}]
Let $X$ be a symplectic manifold.  The Tseng-Yau symplectic cohomology of $X$ is defined as
$$ H^k_{\der + \der^\Lambda}(X) := \frac{\Ker (\der + \der^\Lambda) \cap \Omega^k(X)}{\Image (\der\der^\Lambda) \cap \Omega^k(X)}.$$
\end{defn}

Indeed we have a more refined version for a given real polarization $(U,\Lambda)$ on $X$. 

\begin{defn}
Let $X$ be a symplectic manifold with a real polarization $(U,\Lambda)$.  Define
$$ H^{(p,q)^\Lambda}_{\der + \der^\Lambda}(X) := \frac{\Ker (\der + \der^\Lambda) \cap \Omega^{(p,q)^\Lambda}(X)}{\Image (\der\der^\Lambda) \cap \Omega^{(p,q)^\Lambda}(X)} $$
where we recall that $\Omega^{(p,q)^\Lambda}(X)$ consists of forms whose restrictions to $U$ consist of $p$ $\Lambda$-directions and $q$ $\Lambda^\perp$-directions.
\end{defn}

In the semi-flat setting given in Section \ref{Fourier}, we also have the torus-invariant counterparts
\begin{equation} \label{eq:BC}
H_{B,\textrm{B.-C.}}^{p,q} (\check{X}) := \frac{\Ker (\der) \cap \Omega_B^{p,q}(\check{X})}{\Image (\partial\bar{\partial}) \cap \Omega_B^{p,q}(\check{X})}
\end{equation}
and
\begin{equation} \label{eq:TY}
H^{(p,q)^\Lambda}_{B,\der + \der^\Lambda}(X) := \frac{\Ker (\der + \der^\Lambda) \cap \Omega_B^{(p,q)^\Lambda}(X)}{\Image (\der\der^\Lambda) \cap \Omega_B^{(p,q)^\Lambda}(X)}.
\end{equation}

For a Type IIB supersymmetric $SU(3)$ system $(\check{X},\omega,\Omega)$, suppose that $\omega_t$ for $t \in (-\epsilon,\epsilon)$ is a deformation of the balanced metric $\omega$ (meaning $\omega_{t=0} = \omega$) such that $(\check{X},\omega_t,\Omega)$ is still a Type IIB supersymmetric $SU(3)$ system, with a fixed conformal factor $F$ and fixed type-B flux $\rho_B$.  As argued in \cite[Section 2.1]{TY}, the first order deformation $\delta(\omega^2) = \dd{t}|_{t=0} \omega^2_t$ belongs to $H^{2,2}_{\textrm{B.C.}}(\check{X}) \cap \cL \cdot \cP^{1,1}(\check{X})$, where $\cP^{1,1}(\check{X})$ denotes the space of primitive $(1,1)$ classes and $\cL$ denotes the Lefschetz action $\omega \wedge \cdot$.  It easily generalizes to all dimensions:

\begin{prop}
Let $(\check{X},\omega_t,\Omega)$ for $t \in (-\epsilon,\epsilon)$ be a one-parameter family of Type IIB supersymmetric $SU(n)$ systems with a fixed conformal factor.  Then the infinitesimal deformation $\left.\dd{t}\right|_{t=0} (\omega_t^{n-1})$ represent a class in
$$H^{n-1,n-1}_{\textrm{B.C.}}(\check{X}) \cap \cL^{n-2} \cdot \cP^{1,1}(\check{X})_{\textrm{B.C.}}.$$
\end{prop}

\begin{proof}
The balanced condition $\der (\omega_t^{n-1}) = 0$ implies that $\der \left.\dd{t}\right|_{t=0} (\omega_t^{n-1}) = 0$, and hence $\left.\dd{t}\right|_{t=0} (\omega_t^{n-1})$ represents a class in $H^{n-1,n-1}_{\textrm{B.C.}}(\check{X})$.
Taking $\left.\dd{t}\right|_{t=0}$ on 
$$\check{\Omega} \wedge \bar{\check{\Omega}} = \consti^n \, \check{F} \cdot \frac{\check{\omega}_t^n}{n!}$$
where $\check{F}$ is the conformal factor, we have $\omega^{n-1} \wedge \left.\dd{t}\right|_{t=0} \omega_t = 0$, and hence $\left.\dd{t}\right|_{t=0} \omega_t$ is a primitive $(1,1)$ form.  Thus $\left.\dd{t}\right|_{t=0} (\omega_t^{n-1}) = (n-1) \omega^{n-2} \wedge \left.\dd{t}\right|_{t=0} (\omega_t)$ belongs to $\cL^{n-2} \cdot \cP^{1,1}(\check{X})_{\textrm{B.C.}}$.
\end{proof}

\begin{remark}
Indeed as in \cite{TY} we can say more: if we also fix the Type-B flux of $(\check{X},\omega_t,\Omega)$, then the infinitesimal deformation $\left.\dd{t}\right|_{t=0} (\omega_t^{n-1})$ is actually harmonic with respect to the Bott-Chern cohomology.
\end{remark}

Now consider a Type IIA supersymmetric $SU(n)$ system $(X,\omega,\Omega)$.  Its deformation theory is more tricky than the $n=3$ case because it depends on a real polarization.
 
\begin{prop}
Let $(X,\omega,\Omega_t)$ be a one-parameter family of Type IIA supersymmetric $SU(n)$ systems with a fixed special real polarization $(U,\Lambda)$ on $X$ and a fixed conformal factor.  Then $\pi_{\Lambda}^{1,n-1} \cdot \left.\left(\left.\dd{t}\right|_{t=0} \Omega_t\right)\right|_U$ represents a class in 
$$H^{(1,n-1)^{\Lambda}}_{\der + \der^{\Lambda}}(U) \cap [\pi_{\Lambda}^{1,n-1} \cdot \Omega^{(n-1,1)}(U)]$$
where $\Omega^{(n-1,1)}$ denotes the space of $(n-1,1)$ forms (with respect to the almost complex structure induced from $\Omega = \Omega_{t=0}$).
\end{prop}

\begin{proof}
$\left.\dd{t}\right|_{t=0} \Omega_t$ only has $(n,0)$ and $(n-1,1)$ components.
Taking $\left.\dd{t}\right|_{t=0}$ on $\Omega_t \wedge \bar{\Omega}_t = \consti^n \, F \cdot \frac{\omega^n}{n!}$ where $F$ is the conformal factor, we have $\left(\left.\dd{t}\right|_{t=0} \Omega_t\right) \wedge \bar{\Omega} = 0$, and hence $\left.\dd{t}\right|_{t=0} \Omega_t$ has no $(n,0)$ component, and hence belongs to $\Omega^{(n-1,1)}(X)$.  By the condition $\der (\pi_{\Lambda}^{1,n-1} \cdot \Omega_t|_U) = 0$, we have $\der \left(\pi_{\Lambda}^{1,n-1} \cdot \left(\left.\dd{t}\right|_{t=0} \Omega_t\right)\right) = 0$.  Moreover $\Omega_t$ is a primitive form as $\omega \wedge \Omega_t = 0$, and hence $\left.\dd{t}\right|_{t=0} \Omega_t$ is also primitive.  In particular $\der^{\Lambda} \cdot \left.\dd{t}\right|_{t=0} \Omega_t = 0$.  Thus $\pi_{\Lambda}^{1,n-1} \cdot \left.\dd{t}\right|_{t=0} \Omega_t|_U$ defines a class in $H^{(1,n-1)^{\Lambda}}_{\der + \der^{\Lambda}}(U)$.
\end{proof}

From Section \ref{SYZ}, we see that SYZ and Fourier-Mukai transform gives the mirror map from the Type IIA moduli space to the Type IIB moduli space of the mirror.  Indeed the above cohomologies are also mirror to each other by Fourier-Mukai transform:

\begin{theorem}
Let $(X,\omega) \to B$ be a Lagrangian torus bundle and $(\check{X},\Omega)$ be its mirror.  Fourier-Mukai transform gives an isomorphism
$$ H^{(n-p,q)^\Lambda}_{B,\der + \der^\Lambda}(X,\cpx) \cong H_{B,\textrm{B.C.}}^{p,q} (\check{X}) $$
where $\Lambda$ is the real polarization on $X$ induced from the bundle structure.
\end{theorem}

\begin{proof}
Fourier-Mukai transform maps a $T$-invariant $(p,q)$ form on $\check{X}$ to a $T$-invariant form on $X$ with $(n-p)$ fiber directions and $q$ base directions.  By Theorem \ref{thm:FT}, $\dbar$ and $\partial$ on $\check{X}$ correspond to $\der$ and $\der^\Lambda$ on $X$ respectively under Fourier-Mukai transform.  Hence $d, \dbar\partial$ on $\check{X}$ correspond to $\der + \der^\Lambda$ and $\der\der^\Lambda$ on $X$ respectively.  By Equation \eqref{eq:BC} and \eqref{eq:TY}, statement follows.
\end{proof}

\section{Nilmanifolds} \label{nilmfd}

We describe an important example of a mirror pair of non-K\"ahler supersymmetric Type-A and Type-B systems constructed from nilmanifolds.  We first describe the three-dimensional case, and extend the discussion to all dimensions.

\subsection{Three dimensional case}
The three-dimensional real Iwasawa manifold $B$ is given by the quotient of the Heisenberg group, which consists of matrices of the form
$$ \left( \begin{array}{ccc}
1 & r_1 & r_3 \\
0 & 1 & r_2 \\
0 & 0 & 1
\end{array}
\right), r_1,r_2,r_3 \in \real, $$
by the discrete subgroup $\Gamma$ consisting of matrices of the same form with $r_1,r_2,r_3$ being integers.

Explicitly $\Gamma$ acts on the Heisenberg group by affine transformations: for
$$ \gamma = \left( \begin{array}{ccc}
1 & a & c \\
0 & 1 & b \\
0 & 0 & 1
\end{array}
\right) \in \Gamma $$
we have
$$ \gamma \cdot \left( \begin{array}{c} r_1 \\ r_2 \\ r_3 \end{array} \right) = \left( \begin{array}{ccc}
1 & 0 & 0 \\
0 & 1 & 0 \\
0 & a & 1
\end{array}
\right) \left( \begin{array}{c} r_1 \\ r_2 \\ r_3 \end{array} \right) + \left( \begin{array}{c} a \\ b \\ c \end{array} \right). $$
Thus $B$ is an affine manifold.

Now consider $X = T^*B / \Lambda^*$ and $\check{X} = TB / \Lambda$.
$\check{X}$ has a canonical complex structure, and its complex coordinates are given by $\zeta_i = \check{\theta}_i + \consti r_i, i=1,2,3$.  It has a holomorphic volume form
$$ \check{\Omega} = \der \zeta_1 \wedge \der \zeta_2 \wedge \der \zeta_3. $$
$X$ has a canonical symplectic form
$\omega = \sum_{i=1}^3 \der \theta_i \wedge \der r_i.$

Take the Hermitian two-form
\begin{align*}
\check{\omega} &= \frac{\bi}{2} \left( \der \zeta_1 \wedge \der \bar{\zeta}_1 + \der \zeta_2 \wedge \der \bar{\zeta}_2 + (\der \zeta_3 - r_1 \der \zeta_2) \wedge (\der \bar{\zeta}_3 - r_1 \der \bar{\zeta}_2) \right) \\
&= \left( \der \check{\theta_1} \wedge \der r_1 + \der \check{\theta_2} \wedge \der r_2 + (\der \check{\theta}_3 - r_1 \der \check{\theta}_2) \wedge (\der r_3 - r_1 \der r_2) \right)
\end{align*}
on $\check{X}$.  Notice that $\der \check{\omega} \not= 0$ while $\der \check{\omega}^2 = 0$ by direct computation.  Moreover,
$$\check{\omega}^3 = \der \check{\theta}_1 \wedge \der r_1 \wedge \der \check{\theta}_2 \wedge \der r_2 \wedge \der \check{\theta}_3 \wedge \der r_3$$
which is a constant multiple of $\check{\Omega} \wedge \bar{\check{\Omega}}$.  Thus $(\check{X},\check{\omega}, \check{\Omega})$ forms a type IIB supersymmetric system with a constant conformal factor.

The type-B flux source is
$$\rho_B = 2 \consti \partial \bar{\partial} \check{\omega} \sim \der r_1 \wedge \der \check{\theta}_1 \wedge \der r_2 \wedge \der \check{\theta}_2$$ 
up to some constant multiple.  $\rho_B$ is the Poincar\'e dual of the complex manifold defined by $r_1 = r_2 = \check{\theta}_1 = \check{\theta}_2 = 0$, which can also be written as $TC / (C \cap \Lambda)$ for $C = \{r_1 = r_2 = 0\} \subset B$.

Now take the Fourier-Mukai transform of $\conste^{2\check{\omega}}$.  Write
$$ \check{\omega} = \frac{\bi}{2} \left( \der \zeta_1 \wedge \der \bar{\zeta}_1 + \der \zeta_2 \wedge (\der \bar{\zeta}_2 - r_1 \der \bar{\zeta}_3 + r_1^2 \der \bar{\zeta}_2) + \der \zeta_3 \wedge (\der \bar{\zeta}_3 - r_1 \der \bar{\zeta}_2) \right). $$
Switching polarization (Equation \eqref{switch}) gives
$$ \cP \cdot \check{\omega} = \frac{\bi}{2} \left( \der \check{\theta}_1 \wedge \der r_1 + \der \check{\theta}_2 \wedge (\der r_2 - r_1 \der r_3 + r_1^2 \der r_2) + \der \check{\theta}_3 \wedge (\der r_3 - r_1 \der r_2)\right).$$
Then
\begin{align*}
\conste^{2\check{\omega}} \conste^{\sum_{i=1}^3 \der \check{\theta}_i \wedge \der \theta_i} =& \exp(\der\check{\theta}_1 \wedge (\der \theta_1 + \consti \der r_1) + \der\check{\theta}_2 \wedge (\der \theta_2 + \consti(\der r_2 - r_1 \der r_3 + r_1^2 \der r_2)) \\
&+ \der\check{\theta}_3 \wedge (\der\theta_3 + \consti (\der r_3 - r_1 \der r_2))).
\end{align*}
Thus the Fourier-Mukai transform is
\begin{align*}
\Omega &=  (\der \theta_1 + \consti \der r_1) \wedge (\der\theta_2 + \consti(\der r_2 - r_1 \der r_3 + r_1^2 \der r_2)) \wedge (\der\theta_3 + \consti (\der r_3 - r_1 \der r_2))\\
&= (\der \theta_1 + \consti \der r_1) \wedge ((\der\theta_2 + r_1 \der\theta_3) + \consti \der r_2) \wedge (\der\theta_3 + \consti (\der r_3 - r_1 \der r_2))
\end{align*}
on $X$.
It is easy to verify that $\der (\mathrm{Re} \, \Omega) = 0$ and $\Omega \wedge \bar{\Omega}$ equals to $\omega^3$ up to a constant multiple.  Thus $(X,\omega,\Omega)$ is a Type IIA supersymmetric $SU(n)$ system with a constant conformal factor.

The type-A flux source is
$$\rho_A = \der\der^{\Lambda}(\rIm \,\Omega) \sim \der r_1 \wedge \der r_2 \wedge \der \theta_3$$
up to a constant multiple.  $\rho_A$ is the Poincar\'e dual of the Lagrangian submanifold defined by $r_1 = r_2 = \theta_3 = 0$, which can also be written as the conormal bundle $N^* C / (C \cap \Lambda^*)$, where $C = \{r_1 = r_2 = 0\} \subset B$ is as defined above.  Note that the type-A flux current $\rho_A$ of X and type-B flux current $\rho_B$ of the mirror $\check{X}$ correspond to each other by Fourier-Mukai transform.

\subsection{General dimension}
We may extend the above example to general dimensions to construct type II-A and type II-B supersymmetric $SU(n)$ systems which are mirror to each other.  Consider the vector space $V$ of upper unitriangular matrices $(r_{ij})_{i,j=1}^{K}$, that is, $r_{ij} = 0$ for $i > j$ and $r_{ii}=1$ for all $i=1,\ldots,K$.  Then we consider the lattice $\Gamma$ of upper unitriangular matrices $(a_{ij})_{i,j=1}^{K}$ with integer entries, $a_{ij} = 0$ for $i > j$ and $a_{ii}=1$ for all $i=1,\ldots,K$.  $\Gamma$ acts on $V$ by left multiplication, and the real Iwasawa manifold $B$ is the quotient $V / \Gamma$.  Explicitly $a = (a_{ij})_{i<j} \in \Gamma$ acts on $r = (r_{ij})_{i<j} \in V$ by $a \cdot r = r'$, where
$$ r'_{ik} = r_{ik} + \sum_{j=i+1}^{k-1} a_{ij} r_{jk} + a_{ik} $$
for all $i < k$.  In particular when $k = i+1$, $r'_{i,i+1} = r_{i,i+1} + a_{i,i+1}$ for all $i=1,\ldots,K-1$.

$B$ is an affine manifold.  We have the following global one-forms which will be useful to construct the non-K\"ahler geometries:

\begin{prop} \label{global form}
Define inductively over $k-i\in \N$ the following one forms on $V$:
$$ e_{ik} := \der r_{ik} - \sum_{j=i+1}^{k-1} r_{ij} e_{jk} $$
and $e_{i,i+1} := \der r_{i,i+1}$ for all $i = 1,\ldots,K-1$.
These one forms are invariant under the action of $\Gamma$ and hence descend to be global one-forms on $B = V / \Gamma$.
\end{prop}

\begin{proof}
We prove by induction over $k-i\in \N$.  For $k-i = 1$, since $a = (a_{ij})_{i<j} \in \Gamma$ acts by mapping $r_{i,i+1}$ to $r_{i,i+1} + a_{i,i+1}$, it follows that $e_{i,i+1} = \der r_{i,i+1}$ is invariant under $\Gamma$.  Now consider
$$ e_{ik} = \der r_{ik} - \sum_{j=i+1}^{k-1} r_{ij} e_{jk} $$
which is sent to
\begin{align*}
& (\der r_{ik} + a_{i,i+1} \der r_{i+1,k} + \ldots + a_{i,k-1} \der r_{k-1,k}) \\
& - \sum_{j=i+1}^{k-1} (r_{ij} + a_{i,i+1} r_{i+1,j} + \ldots + a_{i,j-1} r_{j-1,j} + a_{ij}) e_{jk} \\
=& e_{ik} + \left( a_{i,i+1} \der r_{i+1,k} + \ldots + a_{i,k-1} \der r_{k-1,k} - \sum_{j=i+1}^{k-1} (a_{i,i+1} r_{i+1,j} + \ldots + a_{i,j-1} r_{j-1,j} + a_{ij}) e_{jk} \right)
\end{align*}
under the action, because by inductive assumption $e_{jk}$ is invariant for all $j=i+1, \ldots, k-1$.  Consider the second term in the last expression:
\begin{align*}
& a_{i,i+1} \der r_{i+1,k} + \ldots + a_{i,k-1} \der r_{k-1,k} - \sum_{j=i+1}^{k-1} (a_{i,i+1} r_{i+1,j} + \ldots + a_{i,j-1} r_{j-1,j} + a_{ij}) e_{jk} \\
=& \sum_{j=i+1}^{k-1} \left( a_{ij} \der r_{jk} - \left(a_{ij} + \sum_{l=i+1}^{j-1} a_{il}r_{lj} \right) e_{jk} \right) \\
=& \sum_{j=i+1}^{k-1} \left( a_{ij} \sum_{l=j+1}^{k-1} r_{jl} e_{lk} - \left(\sum_{l=i+1}^{j-1} a_{il}r_{lj} \right) e_{jk} \right) \\
=& \sum_{j=i+1}^{k-1}\sum_{l=j+1}^{k-1} a_{ij}r_{jl}e_{lk} - \sum_{j=i+1}^{k-1} \sum_{l=i+1}^{j-1} a_{il}r_{lj} e_{jk} \\
=& 0
\end{align*}
where we have used the equality
$$ \der r_{jk} - e_{jk} = \sum_{l=j+1}^{k-1} r_{jl} e_{lk}$$
by definition of $e_{jk}$.  Thus $e_{ik}$ is invariant under the action.
\end{proof}

Now consider $X = TB/\Lambda$, which is automatically a complex manifold.  Suppose $\theta_{ij}$ are the fiber coordinates of $TB$ corresponding to the coordinates $r_{ij}$ on $B$ for $i < j$.  Then $\theta_{ij}$ transforms in the similar way under the action of $a = (a_{ij})_{i<j} \in \Gamma$:
$$ \theta'_{ik} = \theta_{ik} + \sum_{j=i+1}^{k-1} a_{ij} \theta_{jk}. $$
The global one-forms on $B$ in Proposition \ref{global form} pulls back to be global one-forms on $X$.  We also have the following global one-forms on $X$:
$$ f_{ik} := \der \theta_{ik} - \sum_{j=i+1}^{k-1} r_{ij} f_{jk} $$
defined inductively on $k-i \in \N$ as in the definition of $e_{ik}$'s.

The holomorphic volume form on $X$ is defined as
$$ \Omega = \bigwedge_{i<j} \der z_{ij} $$
where $z_{ij} = \theta_{ij} + \bi r_{ij}$.  It automatically satisfies Equation \eqref{holo}.  (We fix an order, say dictionary order for $\{(i,j):i<j\}$, to define the wedge product in the above expression.)  Take
$$\omega := \sum_{i<j} e_{ij} \wedge f_{ij}$$

\begin{prop}
$\omega$ is a Hermitian $(1,1)$-form on $X$.  Moreover, $\Omega \wedge \bar{\Omega}$ equals to $\omega^n$ up to some constant multiple, that is, the conformal factor is a constant.
\end{prop}
\begin{proof}
Since
$$ \bigwedge_{i<j} e_{ij}  = \bigwedge_{i<j} \der r_{ij} $$
and
$$ \bigwedge_{i<j} f_{ij}  = \bigwedge_{i<j} \der \theta_{ij},$$
it easily follows that $\Omega \wedge \bar{\Omega}$ equals to $\omega^n$ up to some constant multiple.  In particular, $\omega$ is non-degenerate.

To see that $\omega$ is a $(1,1)$-form, it suffices to prove that $e_{ij} \wedge f_{kl} + e_{kl} \wedge f_{ij}$ is a $(1,1)$-form, and we will prove it using induction on $(j-i,l-k)$.  This is true when $j-i = l-k = 1$ since in such a case $e_{ij} = \der r_{ij}$ and $f_{kl} = \der \theta_{kl}$, and $\der r_{ij} \wedge \der \theta_{kl} + \der r_{kl} + \der \theta_{ij} $ is a $(1,1)$-form.  Now consider
\begin{align*}
&e_{ij} \wedge f_{kl} + e_{kl} \wedge f_{ij} \\
=& \left(\der r_{ij} - \sum_{p=i+1}^{j-1} r_{ip} e_{pj}\right) \wedge \left(\der \theta_{kl} - \sum_{q=k+1}^{l-1} r_{kq} f_{ql}\right) \\
&+ \left(\der r_{kl} - \sum_{p=k+1}^{l-1} r_{kp} e_{pl}\right) \wedge \left(\der \theta_{ij} - \sum_{q=i+1}^{j-1} r_{iq} f_{qj}\right).
\end{align*}
We already know that $\der r_{ij} \wedge \der \theta_{kl} + \der r_{kl} \wedge \der \theta_{ij}$ is a $(1,1)$-form.  Moreover,
$$\der r_{ij} \wedge \left(\sum_{q=k+1}^{l-1} r_{kq} f_{ql}\right) + \left(\sum_{p=k+1}^{l-1} r_{kp} e_{pl} \right)\wedge \der \theta_{ij}$$
and
$$\der r_{kl} \wedge \left(\sum_{q=i+1}^{j-1} r_{iq} f_{qj}\right) + \left(\sum_{p=i+1}^{j-1} r_{ip} e_{pj} \right)\wedge \der \theta_{kl}$$
are $(1,1)$-forms, because they are linear combinations of $\der r_{ab} \wedge \der \theta_{cd} + \der r_{cd} \wedge \der \theta_{ab}$'s.  Finally by inductive assumption, $e_{pj} \wedge f_{ql} + e_{ql} \wedge f_{pj}$'s are also $(1,1)$-forms for $j-p < j-i$,$l-q < l-k$.

To prove that $\omega$ is a Hermitian form, it suffices to see that $\omega(\dd{r_{ij}},\dd{\theta_{ij}}) > 0$ for all $i<j$.  Consider
$$ e_{cd} \wedge f_{cd} = \left(\der r_{cd} - \sum_{p=c+1}^{d-1} r_{cp} e_{pd}\right) \wedge \left(\der \theta_{cd} - \sum_{q=c+1}^{d-1} r_{cq} f_{qd}\right). $$
Note that $(e_{ij} \wedge f_{ij})(\dd{r_{ij}},\dd{\theta_{ij}}) = 1$, and all other terms of the form $(e_{cd} \wedge f_{cd})(\dd{r_{ij}},\dd{\theta_{ij}})$  contribute by squares of a function in $r_{ab}$'s which are non-negative.  Thus $\omega$ is a Hermitian form.
\end{proof}

Thus Equation \eqref{oOA} is satisfied with the conformal factor $f$ being a constant.

We will need the following lemma for the behavior of $e_{ij}$ and $f_{ij}$ upon differentiation:

\begin{lemma}
$$ \der e_{ij} = - \sum_{k=i+1}^{j-1} e_{ik} \wedge e_{kj} $$
and
$$ \der f_{ij} = - \sum_{k=i+1}^{j-1} e_{ik} \wedge f_{kj} $$
for all $i<j$.
\end{lemma}
\begin{proof}
We prove by induction on $j - i \in \N$.  When $j - i = 1$, $\der e_{ij} = \der f_{ij} = 0$.  Recall that
$$ e_{ij} = \der r_{ij} - \sum_{l=i+1}^{j-1} r_{il} e_{lj}. $$
Upon differentiation,
\begin{align*}
\der e_{ij} &= - \sum_{l=i+1}^{j-1} \der r_{il} \wedge e_{lj} - \sum_{p=i+1}^{j-1} r_{ip}\der e_{pj} \\
&= - \sum_{l=i+1}^{j-1} \der r_{il} \wedge e_{lj} + \sum_{p=i+1}^{j-1} r_{ip} \left(\sum_{l=p+1}^{j-1} e_{pl} \wedge e_{lj}\right) \\
&= - \sum_{l=i+1}^{j-1} \der r_{il} \wedge e_{lj} + \sum_{l=i+1}^{j-1} \sum_{p=i+1}^{l-1} r_{ip} e_{pl} \wedge e_{lj} \\
&= - \sum_{l=i+1}^{j-1} e_{il} \wedge e_{lj}
\end{align*}
where the second line follows from inductive assumption applied on $\der e_{pj}$ for $j-p < j-i$.  Similarly,
$$ f_{ij} = \der \theta_{ij} - \sum_{l=i+1}^{j-1} r_{il} f_{lj}. $$
Upon differentiation,
\begin{align*}
\der f_{ij} &= - \sum_{l=i+1}^{j-1} \der r_{il} \wedge f_{lj} - \sum_{p=i+1}^{j-1} r_{ip}\der f_{pj} \\
&= - \sum_{l=i+1}^{j-1} \der r_{il} \wedge f_{lj} + \sum_{p=i+1}^{j-1} r_{ip} \left(\sum_{l=p+1}^{j-1} e_{pl} \wedge f_{lj}\right) \\
&= - \sum_{l=i+1}^{j-1} \der r_{il} \wedge f_{lj} + \sum_{l=i+1}^{j-1} \sum_{p=i+1}^{l-1} r_{ip} e_{pl} \wedge f_{lj} \\
&= - \sum_{l=i+1}^{j-1} e_{il} \wedge f_{lj}.
\end{align*}
This proves the required formulas.
\end{proof}

We have
\begin{prop}
$$ \der (\omega^{n-1}) = 0 $$
but
$$ \der (\omega^{n-2}) \not= 0 $$
where $n = K(K-1)/2$ is the (complex) dimension of $X$.  Thus $\omega$ defines a balanced metric on $X$.
\end{prop}

\begin{proof}
\begin{align*}
\der \omega &= \der e_{ij} \wedge f_{ij} - e_{ij} \wedge \der f_{ij} \\
&= - \sum_{k=i+1}^{j-1} e_{ik} \wedge e_{kj} \wedge f_{ij} - \sum_{l=i+1}^{j-1} e_{ij} \wedge e_{il} \wedge f_{lj}.
\end{align*}
All terms in the above expression are linearly independent.  Since each term of $\omega^{n-2}$ contains all but two pairs of $(e_{**},f_{**})$, $\omega^{n-2} \wedge (e_{ik} \wedge e_{kj} \wedge f_{ij})$ must be zero since the index sets $\{i,k\},\{k,j\},\{i,j\}$ are pairwise distinct.  Similarly $\omega^{n-2} \wedge (e_{ij} \wedge e_{il} \wedge f_{lj}) = 0$.  Thus
$$\der (\omega^{n-1}) = (n-1) \omega^{n-2} \wedge \der \omega = 0. $$

On the other hand, each term of $\omega^{n-3}$ contains all but three pairs of $(e_{**},f_{**})$.  Thus for each $e_{ik} \wedge e_{kj} \wedge f_{ij}$, there exists a unique term in $\omega^{n-3}$ such that their product is non-zero.  This also holds for $e_{ij} \wedge e_{il} \wedge f_{lj}$.  Thus
$$\der (\omega^{n-2}) = (n-2) \omega^{n-3} \wedge \der \omega \not= 0.$$
\end{proof}

Thus we see that Equation \eqref{balanced} is satisfied.  The RR flux source $\rho_B$ is $2 \bi \partial \bar{\partial} \omega$, which is a closed four-form.  Then we obtain a Type-$B$ supersymmetric $SU(n)$ structure (up to a constant multiple in the first equation):

\begin{align*}
\Omega \wedge \bar{\Omega} &= \omega^n, \\
\der (\omega^{n-1}) &= 0,\\
2 \consti \partial \bar{\partial} \omega  &= \rho_B,\\
\der \Omega &= 0.
\end{align*}

The mirror of this system is a Type-$A$ supersymmetric $SU(n)$ structure defined on the dual torus bundle $\check{X} = T^*B/\Lambda^*$.  Let $\check{\theta}_1,\ldots,\check{\theta}_n$ be the fiber coordinates on $T^*B$ corresponding to the coordinates $r_1,\ldots,r_n$ on $B$.  Then $r_1,\ldots,r_n,\check{\theta}_1,\ldots,\check{\theta}_n$ descends to be local coordinates on $\check{X}$.  $\check{X}$ is equipped with a canonical symplectic form $$\check{\omega} = \der \check{\theta}_1 \wedge \der r_1 + \ldots + \der \check{\theta}_n \wedge \der r_n.$$

We have the global one-forms $e_{ij}$ for $i<j$ on $B$ defined inductively by Equation \eqref{global form}, which are pulled back to be global one-forms on $\check{X}$.  Also recall the global one-forms $f_{ij}$ defined on $X$ for $i < j$.  The dual basis $\{\check{f}_{ij}: i<j\}$ gives global one-forms on $\check{X}$.  The explicit expressions of $\check{f}_{ij}$s are given by the following proposition:

\begin{prop}
Define $\check{f}_{jk}$ inductively on $j \in \N$ by
\begin{equation}
\check{f}_{jk} := \der \check{\theta}_{jk} + \sum_{i=1}^{j-1} r_{ij} \check{f}_{ik}
\end{equation}
and $\check{f}_{1k} = \der \check{\theta}_{1k}$ for all $k=1,\ldots,K$.  Then we have $\pairing{f_{ij}}{\check{f}_{ab}} = \delta_{ia} \delta_{jb}$, that is, $\{\check{f}_{ij}: i<j\}$ is a dual basis to $\{f_{ij}: i<j\}$.
\end{prop}

\begin{proof}
We prove by induction on $j \in \N$.  For $j=1$, $\check{f}_{1k} = \der \check{\theta}_{1k}$.  Thus
\begin{align*}
\pairing{\check{f}_{1k}}{f_{ab}} &= \pairing{\der \check{\theta}_{1k}}{\der\theta_{ab}} - \sum_{c=a+1}^{b-1} r_{ac} \pairing{\der \check{\theta}_{1k}}{f_{cb}} \\
&= \delta_{1a} \delta_{kb}
\end{align*}
since $f_{cb}$ is a linear combinations of $\der \theta_{pb}$ for $p > c > 1$.

In general
\begin{align*}
\pairing{\check{f}_{jk}}{f_{ab}} &= \pairing{\der \check{\theta}_{jk}}{f_{ab}} + \sum_{i=1}^{j-1} r_{ij} \pairing{\check{f}_{ik}}{f_{ab}}\\
&= \pairing{\der \check{\theta}_{jk}}{\der\theta_{ab}} -\sum_{c=a+1}^{b-1} r_{ac} \pairing{\der \check{\theta}_{jk}}{f_{cb}}  + \sum_{i=1}^{j-1} r_{ij} \pairing{\check{f}_{ik}}{f_{ab}}.
\end{align*}

When $(j,k) = (a,b)$, the first term is $1$; the second term is zero because $f_{cb}$ contains no $\der\theta_{jk}$ term; the third term is zero since $i<j=a$.  Thus $\pairing{\check{f}_{jk}}{f_{jk}} = 1$.

Now consider the case $(j,k) \not= (a,b)$.  The first term is always zero.  If $b \not= k$, the second term is zero automatically because $f_{cb}$ only consists of terms $\der \theta_{pb}$, and the third term is also zero by induction hypothesis on $\check{f}_{ik}$, $i<j$.  When $b=k$, we further divides into two cases: $j<a$ and $j>a$.  When $j<a$, the second and third terms are zero because only $\der \theta_{pb}$ with $p > a > j$ appear in $f_{cb}$ and $f_{ab}$.  When $j>a$, the second term equals to $-r_{aj}$ and by induction hypothesis on $\check{f}_{ik}$ for $i < j$ the third term equals to $r_{aj}$, and hence they cancel each other.  This finishes the proof by induction.
\end{proof}

Using these global one-forms, the complex volume form on $\check{X}$ is defined as
$$ \check{\Omega} = \bigwedge_{j<k} (\check{f}_{jk} + \bi e_{jk}) $$
where again we fix the dictionary order on $\{(j,k):j<k\}$ to define the above product.  It is then a direct verification that $(\check{X},\check{\omega},\check{\Omega})$ satisfies the following system (up to a constant multiple in the first equation)

\begin{align*}
\check{\Omega} \wedge \bar{\check{\Omega}} &= \check{\omega}^n,\\
\der ((\pi_{\Lambda}^{n,0} \oplus \pi_{\Lambda}^{1,n-1})\cdot \check{\Omega}) &= 0,\\
-\consti \der\der^{\Lambda} \cdot (\pi_{\Lambda}^{n-1,1} \cdot \check{\Omega} + \pi_{\Lambda}^{0,n} \cdot \check{\Omega}) &= \check{\rho}_A,\\
\der \check{\omega} &= 0.
\end{align*}

\bibliographystyle{amsalpha}
\bibliography{geometry}

\end{document}